\documentclass[leqno]{amsart}

\usepackage[square, numbers, comma]{natbib} 
\usepackage[utf8x]{inputenc}

\usepackage[usenames,dvipsnames,svgnames,table]{xcolor}
\usepackage{amsmath,amsthm,amssymb,amssymb,esint,verbatim,tabularx,graphicx}
\usepackage{enumerate}
\usepackage{fancyhdr}
\usepackage{pgf,tikz}
\usetikzlibrary{arrows}

\usepackage{hyperref}
\usepackage{pdfpages}
\usepackage{mathtools}

\hypersetup{urlcolor=blue, colorlinks=true} 

\newtheorem{theorem}{Theorem}[section]
\newtheorem{proposition}[theorem]{Proposition}
\newtheorem{lemma}[theorem]{Lemma}
\newtheorem{corollary}[theorem]{Corollary}

\theoremstyle{definition}

\newtheorem{definition}[theorem]{Definition}
\newtheorem{remark}[theorem]{Remark}

\newtheorem{example}[theorem]{Example}

\newcommand{\R}{\ensuremath{\mathbb{R}}}
\newcommand{\N}{\ensuremath{\mathbb{N}}}
\newcommand{\Z}{\ensuremath{\mathbb{Z}}}
\newcommand{\Q}{\mathbb{Q}}

\newcommand{\cL}{\ensuremath{\mathcal{L}}}
\newcommand{\dd}{\,\mathrm{d}}

\newcommand{\veps}{\varepsilon}

\DeclareMathOperator{\supp}{supp}

\begin{document}

\title[Linear operators and the {L}iouville theorem]{The {L}iouville theorem and linear operators satisfying the maximum principle
}

\author[N.~Alibaud]{Natha\"el Alibaud}
\address[N.~Alibaud]{ENSMM\\
26 Chemin de l'Epitaphe\\ 25030 Besan\c{c}on cedex\\ Fran\-ce
and\\
LMB\\ UMR CNRS 6623\\ Universit\'e de Bourgogne Franche-Comt\'e (UBFC)\\ France}
\email{nathael.alibaud\@@{}ens2m.fr}
\urladdr{https://lmb.univ-fcomte.fr/Alibaud-Nathael}

\author[F.~del Teso]{F\'elix del Teso}
\address[F.~del Teso]{Departamento de An\'alisis Matem\'atico y Matem\'atica Aplicada\\
Universidad Complutense de Madrid (UCM)\\
28040 Madrid, Spain} 
\email{fdelteso\@@{}ucm.es}
\urladdr{https://sites.google.com/view/felixdelteso}

\author[J. Endal]{J\o rgen Endal}
\address[J. Endal]{Department of Mathematical Sciences\\
Norwegian University of Science and Technology (NTNU)\\
N-7491 Trondheim, Norway} 
\email{jorgen.endal\@@{}ntnu.no}
\urladdr{http://folk.ntnu.no/jorgeen}

\author[E.~R.~Jakobsen]{Espen R. Jakobsen}
\address[E.~R.~Jakobsen]{Department of Mathematical Sciences\\
Norwegian University of Science and Technology (NTNU)\\
N-7491 Trondheim, Norway} 
\email{espen.jakobsen\@@{}ntnu.no}
\urladdr{http://folk.ntnu.no/erj}



\begin{abstract}
A result by  Courr\`ege  says that linear translation invariant operators satisfy the
maximum  principle  if and only if they are of the form
$\cL=\cL^{\sigma,b}+\cL^\mu$ where
$$
\cL^{\sigma,b}[u](x)=\textup{tr}(\sigma \sigma^{\texttt{T}} D^2u(x))+b\cdot
 Du(x)
$$  
and
$$
\cL^\mu[u](x)=\int_{\R^d\setminus\{0\}} \big(u(x+z)-u(x)-z\cdot Du(x) \mathbf{1}_{|z| \leq
  1}\big) \dd \mu(z).$$
This class of operators coincides with the infinitesimal generators of L\'evy
processes in probability theory.
In this paper we give a complete characterization of the operators of this form that satisfy the Liouville theorem:
Bounded solutions $u$ of $\cL[u]=0$ in $\R^d$ are constant. The
Liouville property is obtained as a consequence of  a periodicity
result that completely characterizes bounded distributional solutions
of $\cL[u]=0$ in $\R^d$. 
The proofs combine arguments from PDEs and group theory. They are simple and short. 
\end{abstract}

\maketitle 

\section{Introduction and main results}
\label{sec:intro}
The classical Liouville theorem states that bounded
solutions of $\Delta u=0$ in $\R^d$ are constant.  The Laplace
operator $\Delta$ is the most classical example of an operator $\cL:C^\infty_\textup{c} (\R^d) \to C(\R^d)$
satisfying the maximum principle in the sense that 
\begin{equation}\label{mp}
\mbox{$\cL [u](x) \leq 0$ at any global maximum point $x$ of $u$.}
\end{equation}
In the class of linear
translation invariant\footnote{Translation invariance means that
  $\cL[u(\cdot+y)](x)=\cL[u](x+y)$ for all $x,y$.} operators
(which includes 
$\Delta$), a result by Courr\`ege \cite{Cou64}\footnote{If \eqref{mp}
holds at any {\it nonnegative} maximum point, then by definition the
{\it positive}  maximum principle holds and by \cite{Cou64} 
there is an extra term $c u(x)$ with $c \leq 0$ in
\eqref{def:localOp1}. For the purpose of this paper (Liouville and periodicity), the case $c<0$ is trivial since then $u=0$
is the unique bounded solution of $\mathcal{L}[u]=0$.} 
says that the maximum
principle holds if and only if 
\begin{equation}\label{eq:GenOp1}
\cL=\cL^{\sigma,b}+\cL^\mu,
\end{equation}
where 
\begin{align}\label{def:localOp1}
\cL^{\sigma,b}[u](x)&=\textup{tr}(\sigma \sigma^{\texttt{T}} D^2u(x))+b \cdot Du(x),\\
\label{def:levy1}
\cL^\mu[u](x)&=\int_{\R^d\setminus\{0\}} \big(u(x+z)-u(x)-z \cdot Du(x) \mathbf{1}_{|z| \leq 1}\big) \dd \mu(z),
\end{align}
and
\begin{align}
& b\in\R^d, \ \ \text{and}
  \ \ \text{$\sigma=(\sigma_1,\ldots,\sigma_P)\in\R^{d\times P}$ for
    $P\in\N$,  $\sigma_j \in\R^d$,}\label{as:sigmab}\tag{$\textup{A}_{\sigma,b}$}\\
&\mu\geq0 \ \text{is a Radon measure on $\R^d\setminus\{0\}$, $\int_{\R^d\setminus\{0\}} \min\{|z|^2,1\} \dd \mu(z)<\infty$.}\label{as:mus}\tag{$\textup{A}_{\mu}$}
\end{align}
These 
 elliptic  operators have a local part $\cL^{\sigma,b}$ and a  nonlocal part
$\cL^\mu$, either of which could be zero.\footnote{The representation
   \eqref{eq:GenOp1}--\eqref{def:localOp1}--\eqref{def:levy1} is 
   unique up to the choice of a cut-off function in
   \eqref{def:levy1} and a square root $\sigma$ of
   $a=\sigma\sigma^\texttt{T}$. In this paper we always use
   $\mathbf{1}_{|z| \leq 1}$ as a cut-off function.}

Another point of view 
of these operators 
comes from probability and
stochastic processes: Every operator mentioned above is the generator
of a L\'evy process, and conversely, every generator of a L\'evy process
is of the form given above. L\'evy processes are Markov
processes with stationary independent increments and are the
prototypical models of noise in 
science, engineering, and finance. Well-known examples are Brownian
motions, Poisson processes, 
stable processes, and various other types of jump processes.

\smallskip

{\em The main contributions of this paper are the following: 

\smallskip

\begin{enumerate}[ \bf 1.]
\item We give necessary and sufficient conditions for $\cL$ to have
the Liouville property: Bounded solutions $u$ of $\cL[u]=0$ in $\R^d$ are
constant.
\smallskip
\item For general $\cL$, we show that all bounded solutions of $\cL[u]=0$ in $\R^d$ are periodic and we  identify  the set of admissible periods.
\end{enumerate}}

Let us  now state our results.   
For a set $S\subseteq \R^d$, we let $G(S)$ denote the smallest additive  subgroup of $\R^d$
containing $S$ and define the subspace $V_S\subseteq \overline{G(S)}$ by
\begin{equation*}
V_S:=\Big\{g \in \overline{G(S)} \ :\ t g \in \overline{G(S)} \mbox{ } \forall t \in \R \Big\}.
\end{equation*} 
Then we take $\supp(\mu)$ to be the support of the measure $\mu$ and
define
\begin{align*}
G_\mu:=\overline{G(\supp(\mu))},\quad V_\mu:=V_{\supp (\mu)},\quad\text{and}\quad
c_\mu:=-\int_{\{|z|\leq1\}\setminus V_\mu} z \dd\mu(z).
\end{align*} 
Here $c_\mu$ is well-defined and uniquely determined by $\mu$, cf. Proposition \ref{def-prop}.  We also need the subspace
$W_{\sigma,b+c_\mu}:=\textup{span}_{\R}\{\sigma_1,\ldots,\sigma_P,b+c_\mu\}$. 
\begin{theorem}[General Liouville]\label{thm:Liouville}
Assume \eqref{as:sigmab} and \eqref{as:mus}.  Let $\cL$ be given by \eqref{eq:GenOp1}--\eqref{def:localOp1}--\eqref{def:levy1}.  Then the
following statements are equivalent:
\smallskip
\begin{enumerate}[\rm(a)]
\item\label{a4} 
If $u\in L^\infty(\R^d)$ satisfies
  $\cL[u]=0$ in $\mathcal{D}'(\R^d)$, then $u$ is a.e. a constant.
  \smallskip
\item\label{label:thmLiou:c} 
$\overline{G_\mu+W_{\sigma,b+c_\mu}}=\R^d$. 
\end{enumerate}
\end{theorem}

 The  above Liouville result is a consequence of  a periodicity result
for bounded solutions of $\cL[u]=0$ 
in $\R^d$. For a set $S\subseteq \R^d$, a function $u\in L^\infty(\R^d)$ is a.e. 
\emph{$S$-periodic} if 
$u(\cdot+s)=u(\cdot)$ in $\mathcal{D}'(\R^d)$   $\forall s\in S$. Our
result is the following:
\begin{theorem}[General periodicity]\label{thm:PeriodGeneralOp}
Assume \eqref{as:sigmab}, \eqref{as:mus}, and $u\in
L^\infty(\R^d)$. Let $\cL$ be given by \eqref{eq:GenOp1}--\eqref{def:localOp1}--\eqref{def:levy1}. Then the
following statements are equivalent:
\begin{enumerate}[\rm(a)]
\item
\label{a3}
$\cL[u]=0$ in $\mathcal{D}'(\R^d)$.
 \smallskip 
\item\label{b3} $u$ is a.e. $\overline{G_\mu+W_{\sigma,b+c_\mu}}$-periodic. 
\end{enumerate}
\end{theorem}
This result characterizes the bounded solutions for all operators
$\cL$ in our class, also those not satisfying the Liouville
property. Note that if  $\overline{G_\mu+W_{\sigma,b+c_\mu}}=\R^d$, then $u$
is constant and the Liouville result follows. Both theorems  are proved in Section \ref{sec:periodandliou}.

We give examples 
in
Section \ref{sec:examples}. Examples \ref{ex1} and \ref{ex2} provide an overview of different possibilities,
and Examples \ref{ex:finitenumberofpoints} and \ref{ex:kro} are concerned with the case where $\textup{card} \left(\textup{supp} (\mu) \right)<\infty$. The Liouville property holds in the latter case if and only if $\textup{card} \left( \textup{supp} (\mu) \right) \geq d-\textup{dim} \left(W_{\sigma,b+c_\mu} \right)+1$ with additional algebraic conditions in relation with Diophantine approximation. The Kronecker theorem (Theorem
\ref{thm:CharKron}) is a key ingredient in this discussion and
a slight change
in the data may destroy the Liouville property.

The class of operators $\cL$ given by \eqref{eq:GenOp1}--\eqref{def:localOp1}--\eqref{def:levy1} is large  and
diverse. 
In addition to the processes
mentioned above, 
 it 
includes also discrete random
walks, constant coefficient It\^{o}- and L\'evy-It\^{o} processes, and most
processes used as driving noise in finance.
 Examples  of nonlocal operators 
 are  fractional
  Laplacians \cite{Lan72},  convolution 
  operators  \cite{Cov08,A-VMaRoT-M10,BrChQu12},
relativistic Schr\"odinger operators \cite{FaWe16}, 
and the
CGMY model in finance \cite{CoTa04}.  
We mention that discrete  finite difference  operators can be written
in the form
\eqref{eq:GenOp1}--\eqref{def:localOp1}--\eqref{def:levy1},
cf. \cite{DTEnJa18b}. For more examples, see Section
\ref{sec:examples}.   

There is a huge literature on the Liouville  theorem.  
In the local case, we simply refer to the survey \cite{Far07}. In the nonlocal case, the Liouville theorem  is more or less understood
for  fractional Laplacians  or variants 
\cite{Lan72,BoKuNo02,CaSi14, ChDALi15, Fal15}, 
certain L\'evy operators 
\cite{BaBaGu00,PrZa04, ScWa12, R-OSe16a,DTEnJa17a}, 
relativistic Schr\"{o}dinger operators \cite{FaWe16}, or convolution
operators \cite{CD60,BrChQu12, BrCo18, BrCoHaVa19}. The techniques vary from Fourier analysis, potential theory, probabilistic methods, to classical PDE arguments. 

To prove that solutions of $\cL[u]=0$ are $G_\mu$-periodic, we rely on propagation of maximum points
\cite{CD60,Cov08,Cio12,DTEnJa17b,DTEnJa17a,HuDuWu18,BrCo18, BrCoHaVa19} and a
localization technique \`a la
\cite{CD60,BeHaRo07,Ros09,BrCoHaVa19}. As far as we know, Choquet and
Deny \cite{CD60} were the first to obtain such results. They were concerned with
the equation $u \ast \mu-u=0$
for some bounded measure $\mu$. This is a particular case of our equation since
$u \ast \mu-u=\cL^\mu[u]+\int_{\R^d \setminus \{0\}} z  \mathbf{1}_{|z| \leq 1} \dd \mu(z) \cdot Du$. 
For general $\mu$, the drift $\int_{\R^d \setminus \{0\}} z  \mathbf{1}_{|z| \leq 1} \dd \mu(z) \cdot Du$ may not make sense and
the identification of the full drift $b+c_\mu$ relies on a standard decomposition of closed subgroups of $\R^d$, see e.g. \cite{Mar03}. The idea is to establish $G_\mu$-periodicity of solutions of $\cL[u]=0$ as in \cite{CD60}, and then use that $G_\mu=V_\mu \oplus \Lambda$ for the vector space $V_\mu$ previously defined and some discrete group $\Lambda$. This will roughly speaking remove the singularity $z=0 \in V_\mu$ in the computation of $c_\mu$ because $\int_{\R^d \setminus \{0\}}\mathbf{1}_{z \in V_\mu}z \mathbf{1}_{|z| \leq 1} \dd \mu(z) \cdot Du=0$ for any $G_\mu$-periodic function. See Section \ref{sec:periodandliou} for details.

Our approach then combines PDEs and group arguments, extends the results of \cite{CD60} to Courr\`ege/L\'evy operators,
yields necessary and sufficient conditions for the Liouville property, and provides short and simple proofs.

\subsubsection*{Outline of the paper} 
Our main results (Theorems \ref{thm:Liouville} and \ref{thm:PeriodGeneralOp}.) were stated in Section \ref{sec:intro}. They are proved in Section  \ref{sec:periodandliou} and examples are given in Section \ref{sec:examples}.

\subsubsection*{Notation and preliminaries}
 The {\it support} of a measure
 $\mu$
 is defined as
\begin{equation}\label{def-support}
{\supp(\mu)} := \left\{z\in  \R^d \setminus \{0\}  \ : \ \mu(B_r(z))>0, \ \forall r>0\right\},
\end{equation}
 where $B_r(z)$ is the ball of center $z$
and  radius $r$.  
 To continue, we assume \eqref{as:sigmab}, \eqref{as:mus}, and $\cL$ is given by \eqref{eq:GenOp1}--\eqref{def:localOp1}--\eqref{def:levy1}. 

\begin{definition}
For  any $u \in L^\infty( \R^d )$, $\mathcal{L}[u] \in \mathcal{D}'(\R^d)$ is defined by 
\begin{equation*}
\langle \cL[u],\psi \rangle:=\int_{\R^d }u(x)\cL^*[\psi](x)\dd x \quad  \forall \psi\in C_\textup{c}^\infty( \R^d )
\end{equation*}
with $\cL^*:=\cL^{\sigma,-b}+\cL^{\mu^*}$ and $\dd \mu^*(z):= \dd \mu(-z)$. 
\end{definition}

 The above distribution is well-defined since  $\cL^*:W^{2,1}(\R^d) \to L^1(\R^d)$ is bounded.

\begin{definition}
Let $S \subseteq \R^d$ and $u \in L^\infty(\R^d)$, then $u$  
is a.e. \emph{$S$-periodic} if 
\[
\int_{\R^d} \big(u(x+s)-u(x)\big) \psi(x) \dd x=0 \quad \forall s\in S, \forall \psi \in C^\infty_\textup{c}(\R^d).
\]
\end{definition}

The following technical result will be needed to regularize
 distributional solutions of $\cL[u]=0$ and a.e. periodic
 functions. Let the mollifier
$\rho_\veps(x):=\frac1{\veps^{d}}\rho(\frac x\veps)$, $\veps>0$, for
 some $0\leq\rho\in
C_{\textup{c}}^\infty(\R^d)$ with $\int_{\R^d}\rho =1$.

\begin{lemma}\label{lem:smoothReduction} 
 Let $u\in
L^\infty(\R^d)$ and $u_\veps:=\rho_\veps*u$. Then:  
\smallskip
\begin{enumerate}[{\rm (a)}]
\item\label{a1} $\cL[u]=0$ in $\mathcal{D}'(\R^d)$ if and only if
  $\cL[u_\veps]=0$ in $\R^d$ for all $\veps>0$.\smallskip
\item\label{b1} $u$ is a.e. $S$-periodic if and only if $u_\veps$ is $S$-periodic for all $\veps>0$.
\end{enumerate}
\end{lemma}

\begin{proof}
 The proof of \eqref{a1} is standard since $\cL[u_\veps]=\cL[u] \ast \rho _\veps$ in $\mathcal{D}'(\R^d)$.
Moreover \eqref{b1}  follows from \eqref{a1} since for any $s\in S$ we can
take $\cL[\phi](x)=  \phi(x+s)-\phi(x)$ by choosing $\sigma,b=0$ and
$\mu=\delta_s$ (the Dirac measure at $s$) in \eqref{eq:GenOp1}--\eqref{def:localOp1}--\eqref{def:levy1}.
\end{proof}

\section{Proofs}\label{sec:periodandliou}

This section is devoted to the proofs of Theorems \ref{thm:Liouville} and \ref{thm:PeriodGeneralOp}. We first reformulate the classical Liouville theorem for local operators in terms of periodicity, then study the influence of the nonlocal part. 

\subsection{ $W_{\sigma,b}$-periodicity for local operators}

 Let us recall the Liouville theorem  for 
operators of the form \eqref{def:localOp1}, see e.g. \cite{Nel61,Miy15}.  
In
the result we use the set
\begin{equation*}
W_{\sigma,b}=\textup{span}_{\R}\{\sigma_1,\ldots,\sigma_P,b\}.
\end{equation*}
Note that $\textup{span}_{\R}\{\sigma_1,\ldots,\sigma_P\}$ equals the
span of the eigenvectors of $\sigma\sigma^\texttt{T}$ corresponding to nonzero eigenvalues. 
\begin{theorem}[Liouville for $\cL^{\sigma,b}$]\label{thm:LiouvilleLocal}
Assume \eqref{as:sigmab} and $\cL^{\sigma,b}$ is given by
\eqref{def:localOp1}. Then the following statements are equivalent: 
\smallskip
\begin{enumerate}[{\rm (a)}]
\item If $u\in L^\infty(\R^d)$ solves $\cL^{\sigma,b}[u]=0$ in
  $\mathcal{D}'(\R^d)$, then $u$ is a.e. constant in~$\R^d$.
  \smallskip
\item $W_{\sigma,b}=\R^d$.
\end{enumerate}
\end{theorem}
Let us now  reformulate and prove  this classical result as a consequence of a periodicity result, a type
of argument that will be crucial in the nonlocal case. We will consider $C^\infty_{\textup{b}}(\R^d)$ solutions, which will be enough later during the proofs of Theorem \ref{thm:Liouville} and \ref{thm:PeriodGeneralOp}, thanks to Lemma \ref{lem:smoothReduction}.

\begin{proposition}[Periodicity for $\cL^{\sigma,b}$]\label{thm:PeriodLiouvilleLocal}
Assume \eqref{as:sigmab}, $\cL^{\sigma,b}$ is given by
\eqref{def:localOp1}, and $u\in C^\infty_{\textup{b}}(\R^d)$. Then the following
statements are
equivalent: 
\begin{enumerate}[{\rm (a)}]
  \smallskip
\item\label{a2} $\cL^{\sigma,b}[u]=0$ in $\R^d$.
  \smallskip
\item\label{b2} $u$ is $W_{\sigma,b}$-periodic.  
\end{enumerate}
\end{proposition}
 Note  that part \eqref{b2} implies that $u$ is constant in the
directions defined by the vectors $\sigma_1,\ldots,\sigma_P,b$. If
their span then covers all of $\R^d$, Theorem \ref{thm:LiouvilleLocal} follows trivially.
To prove Proposition \ref{thm:PeriodLiouvilleLocal}, we adapt the ideas of \cite{Miy15} to our setting.

\begin{proof}[Proof of Proposition \ref{thm:PeriodLiouvilleLocal}]\ 

\smallskip

\noindent\eqref{b2} $\Rightarrow$ \eqref{a2} \ We have $b \cdot Du(x)=\frac{\dd}{\dd t} u(x+t b)_{|_{t=0}}=0$ for any $x \in \R^d$ since the function $t \mapsto u(x+t b)$ is constant. Similarly $(\sigma_j \cdot D)^2 u(x):=\frac{\dd^2}{\dd t^2} u(x+t \sigma_j)_{|_{t=0}}=0$ for any $j=1,\dots,P$.
Using then that $\textup{tr}(\sigma\sigma^\texttt{T} D^2u)=
\sum_{j=1}^P
(\sigma_j \cdot
D)^2 u$, we conclude that $\cL^{\sigma,b}[u]=0$ in $\R^d$.  

\medskip

\noindent\eqref{a2} $\Rightarrow$ \eqref{b2} \ Let
$v(x,y,t):=u(x+\sigma y-bt)$  for $x \in \R^d$, $y\in\R^P$, and $t \in \R$. Direct computations show that
$$
\Delta_y v(x,y,t)=\sum_{j=1}^P(\sigma_j\cdot D)^2u(x+\sigma y-bt)=\textup{tr}\big[\sigma \sigma^{\texttt{T}} D^2u(x+\sigma y-bt)\big]
$$
 and $ \partial_tv(x,y,t)=-b \cdot Du(x+\sigma y-bt)$. Hence for all
$(x,y,t)\in \R^d \times \R^P \times \R$,
$$
\Delta_y
v(x,y,t)-\partial_tv(x,y,t)=\mathcal{L}^{\sigma,b}[u](x+\sigma y-bt)=0.
$$

Since $v(x,\cdot,\cdot)$ is
bounded, we conclude by uniqueness of the heat equation that  for any
$s<t$,
\begin{equation}\label{eq:convFormulaHE}
v(x,y,t)=\int_{\R^P}v(x,z,s)K_P(y-z,t-s)\dd z,
\end{equation}
where $K_P$ is the standard heat kernel in $\R^P$. But then
$$
\|\Delta_yv(x,\cdot,t)\|_\infty\leq \|v(x,\cdot,s)\|_\infty\|\Delta_y K_P(\cdot,t-s)\|_{L^1(\R^P)},
$$
and since $\|\Delta_y K_P(\cdot,t-s)\|_{L^1}\to0$ as $s\to-\infty$,
 we deduce that
$\Delta_yv=0$ for all $x,y,t$. 

By the classical
Liouville theorem (see e.g. \cite{Nel61}), $v$ is constant in
$y$.  It is also constant in $t$ by \eqref{eq:convFormulaHE}
since $\int_{\R^P}K_P(z,t-s)\dd z=1$. We  conclude that $u$ is $W_{\sigma,b}$-periodic since 
$$
u(x)=v(x,0,0)=v(x,y,t)=u(x+\sigma y-bt)
$$
 and $W_{\sigma,b}=\{\sigma y-bt:y \in \R^P, t \in \R\}$. 
\end{proof}

\subsection{ $G_\mu$-periodicity for general operators}

Proposition \ref{thm:PeriodLiouvilleLocal} might seem artificial in the
local case, but not so in the nonlocal case. In fact we will prove our general
Liouville result as a consequence of a periodicity result. A key step
in this direction is the  lemma below.

\begin{lemma}\label{lem:L0ImpusuppmuPer}
Assume \eqref{as:sigmab}, \eqref{as:mus}, $\cL$ is given by \eqref{eq:GenOp1}--\eqref{def:localOp1}--\eqref{def:levy1}, and $u\in C^\infty_{\textup{b}}(\R^d)$. If $\cL[u]=0$ in $\R^d$, then $u$ is $\supp(\mu)$-periodic.
\end{lemma}

To prove this result, we use propagation of maximum
(see e.g. \cite{CD60,Cov08,Cio12}).
\begin{lemma}\label{max-prop}  If $u\in C^\infty_{\textup{b}}(\R^d)$
achieves its global maximum at some $\bar x$ such that $\cL[u](\bar x)\geq 0$, then $u(\bar x+z)=u(\bar x)$ for any $z \in \textup{supp} (\mu)$. 
  \end{lemma}
\begin{proof}
  At $\bar x$, $u=\sup u$, $Du=0$ and $D^2u\leq 0$, and hence
  $\cL^{\sigma,b}[u](\bar x)\leq 0$ and
  $$0 \leq\cL[u](\bar x) \leq \cL^{\mu}[u](\bar x) =  \int_{\R^d \setminus \{0\}} \big(u(\bar
  x+z) - \sup_{\R^d} u\big)\, d\mu(z).$$
 Using that $\int_{\R^d\setminus\{0\}} f \dd \mu \geq 0$ and $f \leq 0$ implies $f=0$ $\mu$-a.e., we deduce that $u(\bar x+z)-\sup_{\R^d} u=0$ for $\mu$-a.e. $z$. Since $u$ is continuous, this equality holds for all $z \in \textup{supp} (\mu)$.\footnote{ If not, we would find some $z_0$ and $r_0>0$ such that $f(z):=u(\bar x+z)-\sup u<0$ in $B_{r_0}(z_0)$ where as $\mu(B_{r_0}(z_0))>0$ by \eqref{def-support}.} 
\end{proof}

To exploit Lemma \ref{max-prop}, we need to have a maximum point. 
For this sake, we use a localization technique \`a la \cite{CD60,BeHaRo07,Ros09,BrCoHaVa19}. 

\begin{proof}[Proof of  Lemma \ref{lem:L0ImpusuppmuPer}]
 Fix an arbitrary $ \bar z  \in \supp(\mu)$, define
\[
v(x):=u(x+ \bar z )-u(x),
\]
and let us show that  $v(x)=0$  for all $x \in \R^d$. 
We first
 show that $v \leq 0$. Take $M$ and a sequence $\{x_n\}_{n}$ such that 
\[
v(x_n)\stackrel{n\to \infty}{\longrightarrow} M:=\sup v,
\]
and define
$$u_n(x):=u(x+x_n)\quad\text{and}\quad v_n(x):=v(x+x_n).$$  Note that
$\cL[v_n] = 0$ in $\R^d$. Now since $v \in  C_\textup{b}^\infty(\R^d)$, the
Arzel\`a-Ascoli theorem implies that there exists $v_\infty$ such that
$v_n \to v_\infty$ locally uniformly (up to a subsequence). Taking
another subsequence if necessary, we can assume that the derivatives
up to second order converge and pass to the limit in the equation
$\cL[v_n] = 0$ to deduce that $\cL[v_\infty] = 0$ in $\R^d$.
Moreover, $v_\infty$ attains its maximum at $x=0$ since $v_\infty\leq
M$ and
\[
v_\infty(0)=\lim_{n\to\infty}v_n(0)=\lim_{n\to \infty}v(x_n)=M.
\]
A similar argument shows that there is a $u_\infty$ such that $u_n \to
u_\infty$ as $n \to \infty$ locally uniformly. Taking further subsequences
if necessary, we can assume that  $u_n$ and $v_n$ converge along
the same sequence.  Then by construction
$$
v_\infty(x)=u_\infty(x+ \bar z )-u_\infty(x).
$$
By Lemma \ref{max-prop} and an iteration, we find that
$M=v_\infty(m \bar z )=u_\infty((m+1) \bar z )-u_\infty(m \bar z )$ for any $m \in
\Z$. Then by another
iteration,
\[
u_\infty((m+1) \bar z )=u_\infty(m \bar z )+M=\ldots = u_\infty(0)+(m+1)M.
\]
But since $u_\infty$ is bounded, the only choice is $M=0$ and thus $v\leq M=0$.  A similar argument shows that $v \geq 0$, and hence, $0=v(x)=u(x+ \bar z )-u(x)$ for any $ \bar z \in \supp(\mu)$ and all $x\in\R^d$.
\end{proof}

We can give a more general result than Lemma \ref{lem:L0ImpusuppmuPer} if we consider groups.

\begin{definition}
\begin{enumerate}[{\rm (a)}]
\item A set
$G \subseteq \R^d$ is an {\it additive subgroup} if $G \neq \emptyset$ and 
$$
\forall g_1,g_2 \in G, \quad g_1+g_2\in G \quad \text{and}\quad-g_1 \in G. 
$$
\item The \textit{subgroup generated} by a set $S \subseteq \R^d$, denoted $G(S)$, is the smallest additive group containing $S$.
\end{enumerate}
\end{definition}

Now we return to
a key set for our  analysis: 
\begin{equation}\label{def-Gmu}
G_\mu=\overline{G(\supp (\mu))}.
\end{equation}
This set appears naturally because  of the elementary result below. 

\begin{lemma}\label{suppmuPerGsuppmuPer}
Let $S \subseteq \R^d$. Then $w\in C(\R^d)$ is $S$-periodic if and only if $w$ is  $\overline{G(S)}$-periodic. 
\end{lemma}

\begin{proof}
 It suffices to show that $G:=\{g \in \R^d:w(\cdot+g)=w(\cdot)\}$ is a closed subgroup of $\R^d$. It is obvious that it is closed by continuity of $w$. Moreover, for any $g_1,g_2 \in \R^d$ and $x \in \R^d$,
\begin{equation*}
w(x+g_1-g_2)=w(x-g_2)=w(x-g_2+g_2)=w(x).\qedhere
\end{equation*}
\end{proof}

By Lemmas 
\ref{lem:L0ImpusuppmuPer} and \ref{suppmuPerGsuppmuPer}, we have proved that:

\begin{proposition}[$G_\mu$-periodicity]\label{lem:L0ImpusuppmuPer-bis}
Assume \eqref{as:sigmab}, \eqref{as:mus}, $\cL$ is given by \eqref{eq:GenOp1}--\eqref{def:localOp1}--\eqref{def:levy1}, and $G_\mu$ by \eqref{def-Gmu}. Then any solution $u\in C^\infty_{\textup{b}}(\R^d)$ of $\cL[u]=0$ in $\R^d$ is $G_\mu$-periodic.
\end{proposition}

\subsection{The role of $c_\mu$}

 Propositions \ref{thm:PeriodLiouvilleLocal} and
 \ref{lem:L0ImpusuppmuPer-bis} combined may seem to imply that $\cL[u]=0$ gives
 $(G_\mu+W_{\sigma,b})$-periodicity of $u$, but this is not true in
 general. The correct periodicity result depends on a new
 drift $b+c_\mu$, where $c_\mu$ is defined in \eqref{cmu}
 below. To give this definition, we need to decompose 
 $G_\mu$ into 
 a direct sum of a vector subspace and a relative lattice. 

\begin{definition}
\begin{enumerate}[{\rm (a)}]
\item If two subgroups $G,\tilde{G} \subseteq \R^d$ satisfy $G \cap \tilde{G}=\{0\}$, their sum is said to be {\em direct} and we write  $G+\tilde{G}=G \oplus \tilde{G}$.
\smallskip
\item A  \textit{full lattice}  is a subgroup $\Lambda \subseteq \R^d$ of the form 
$
\Lambda=
\oplus_{n=1}^d a_n \Z
$
for some basis $\{a_1,\dots,a_d\}$ of $\R^d$.
A \textit{relative lattice} is a lattice of a vector subspace of $\R^d$.
\end{enumerate}
\end{definition}

\begin{theorem}[Theorem 1.1.2 in \cite{Mar03}]\label{decom:opt}
 If $G$  is a closed subgroup of $\R^d$, then
  $
G=V \oplus \Lambda
$
for some vector space $V \subseteq \R^d$ and some relative lattice $\Lambda  \subseteq \R^d$ such that $V \cap \textup{span}_\R \Lambda=\{0\}$.
\end{theorem}

In this decomposition the space $V$ is unique and can be
represented by \eqref{def-V} below.

\begin{lemma}\label{dis}
Let $V$ be a vector subspace and $\Lambda$ a relative lattice of $\R^d$ such that $V \cap\textup{span}_\R \Lambda =\{0\}$. Then for any $\lambda \in \Lambda$, there is an open ball $B$ of $\R^d$ containing $\lambda$ such that $B \cap (V \oplus \Lambda)=B \cap (V+\lambda)$.
\end{lemma}

\begin{proof}
If the lemma does not hold, there exists $v_n+\lambda_n \to  \lambda$
as $n \to \infty$ where $v_n \in V$, $\lambda_n \in \Lambda$,
$\lambda_n \neq \lambda$. Note that $v_n, \lambda_n,
\lambda\in V \oplus \textup{span}_\R \Lambda$, and that
$$
\lambda=\!\!\underset{\ \, \in V}0+\!\!\underset{\ \, \in \Lambda}{\lambda}.
$$ 
By continuity of the projection from $V \oplus \textup{span}_\R\Lambda$ onto $\textup{span}_\R \Lambda$, $\lambda_n \to \lambda$ and this contradicts the fact that  each point of $\Lambda$ is isolated.
\end{proof}

\begin{lemma}\label{pro:def-V}
Let $G$, $V$ and $\Lambda$ be as in Theorem \ref{decom:opt}. Then
\begin{equation}
\label{def-V}
V=V_G:=\left\{g \in G \ :\ t g \in G \mbox{ } \forall t \in \R  \right\}.
\end{equation}
\end{lemma}

\begin{proof}
It is clear that $V \subseteq V_G$. Now given $g \in V_G$, there is
$(v,\lambda) \in V \times \Lambda$ such that
$g=v+\lambda$. For any  $t \in \R$, $t g=t v+t \lambda \in G$ and thus
$t \lambda \in G$ since $t v \in V \subseteq G$. Let $B$ be an open
ball containing $\lambda$ such that $B \cap G=B \cap
(V+\lambda)$. Choosing $ t$ such that $t \neq 1$ and $t \lambda \in
B$, we infer that $t \lambda={ \tilde{v} }+\lambda$ for some ${
  \tilde{v} } \in V$. Hence $\lambda=(t-1)^{-1} { \tilde{v} } \in V$
and this implies that $\lambda=0$.
In other words $V_G \subseteq V$, and the proof is complete.
\end{proof}

\begin{remark}\label{rem:per} 
Any $G$-periodic function $w \in C^1(\R^d)$ is such that $z \cdot Dw (x)=\lim_{t \to 0} \frac{w(x+t { z })-w(x)}{t}=0$ for any $x \in \R^d$   and $z \in V_G$.
 \end{remark}

By Theorem \ref{decom:opt} and Lemma
\ref{pro:def-V}, we decompose the set $G_\mu$ in
\eqref{def-Gmu} into a lattice and the subspace
$V_\mu:=V_{G_\mu}$. The new drift can then be defined as
\begin{equation}\label{cmu}
c_\mu=-\int_{\{|z| \leq 1\} \setminus V_\mu} z\, \dd \mu(z).
\end{equation}

\begin{proposition}\label{def-prop}
 Assume \eqref{as:mus} and $c_\mu$ is given by
 \eqref{cmu}. Then $c_\mu \in \R^d$ is well-defined
 and uniquely determined by $\mu$.
\end{proposition}

\begin{proof}
Using that $\textup{supp}(\mu) \subset G_\mu=V_\mu\oplus \Lambda$, 
\begin{equation*}
\begin{split}
\int_{\{|z| \leq 1\} \setminus V_\mu} |z| \dd \mu(z) & = \int_{G_\mu \setminus (V_\mu+0)} |z| \mathbf{1}_{|z| \leq 1} \dd \mu(z)\\
& \leq \int_{G_\mu \setminus B} |z| \mathbf{1}_{|z| \leq 1} \dd \mu(z)
\end{split}
\end{equation*}
for some open ball $B$ containing $0$ given by Lemma \ref{dis}. This integral is finite by \eqref{as:mus} which completes the proof.
\end{proof}

\begin{proposition}\label{lem:drift}
Assume \eqref{as:mus} and $\cL^\mu$,   $G_\mu$,  $c_\mu$ are given by
\eqref{def:levy1},  \eqref{def-Gmu},  \eqref{cmu}. If $w\in C^\infty_{\textup{b}}(\R^d)$
is 
$G_\mu$-periodic, 
then
  $$\cL^\mu [w] = c_\mu\cdot Dw \quad \text{in}\quad \R^d.$$
 \end{proposition}

\begin{proof}
 Using that $\int_{\R^d\setminus\{0\}} f \dd \mu=\int_{\textup{supp} (\mu)} f \dd \mu$, we have
\begin{align*}
  \cL^\mu[w](x) 
  = - \int_{\R^d \setminus \{0\}} z\cdot Dw(x)
  \mathbf{1}_{|z| \leq 1}\dd\mu(z)
\end{align*}
because $w(x+z)-w(x)=0$ for all $x \in \R^d$ and $z \in \textup{supp}(\mu) \subset G_\mu$. The result is thus immediate from Remark \ref{rem:per} and Proposition \ref{def-prop}.
\end{proof}

\subsection{ Proofs of Theorems \ref{thm:Liouville} and \ref{thm:PeriodGeneralOp}}

We are now in a position to prove our main results. We start with 
Theorem \ref{thm:PeriodGeneralOp} which characterizes
all bounded solutions of $\cL[u]=0$ in $\R^d$ as periodic functions
and specifies the set of admissible periods.

\begin{proof}[Proof of Theorem \ref{thm:PeriodGeneralOp}]
By Lemma \ref{lem:smoothReduction} we can assume that $u \in
C_\textup{b}^\infty(\R^d)$. 

\medskip

\noindent\eqref{a3} $\Rightarrow$ \eqref{b3} \ Since $\cL[u] =0$ in $\R^d$, $u$
is  $G_\mu$-periodic by 
Proposition \ref{lem:L0ImpusuppmuPer-bis}.
Proposition \ref{lem:drift} then implies that 
$$
0=\cL[u] = \cL^{\sigma,b}[u]+c_\mu\cdot Du=
\cL^{\sigma,b+c_\mu}[u]\quad \text{in}\quad \R^d,
$$
which by Proposition  \ref{thm:PeriodLiouvilleLocal} shows that $u$ is also
$W_{\sigma,b+c_\mu}$-periodic. It is now easy to see that $u$ is $\overline{G_\mu+W_{\sigma,b+c_\mu}}$-periodic.
\medskip

\noindent\eqref{b3} $\Rightarrow$ \eqref{a3} \ Since $u$ is both $G_\mu$ and
$W_{\sigma,b+c_\mu}$-periodic, by first applying  Proposition \ref{lem:drift} and then
Proposition \ref{thm:PeriodLiouvilleLocal},  
$\cL[ u]
=\cL^{\sigma,b+c_\mu}[u]=0$ in $\R^d$. 
\end{proof}

We now prove Theorem \ref{thm:Liouville} on necessary and sufficient
conditions for $\cL$ to satisfy the Liouville
property. We will use the following consequence of Theorem \ref{decom:opt}.

\begin{corollary}\label{pro:group-multid}  
A subgroup $G$ of $\R^d$ is dense if and only if there are no $c \in
\R^d$ and codimension 1  subspace  $H \subset \R^d$ such that $G \subseteq H+c
\Z$.
\end{corollary}

\begin{proof} 
Let us argue by contraposition for both the ``only if'' and ``if'' parts.

\medskip

\noindent ($\Rightarrow$) \ Assume $G \subseteq H+c \Z$ for some codimension 1 space $H$ and $c \in \R^d$. If $c \in H$, then $\overline{G} \subseteq \overline{H}=H \neq \R^d$. If $c \notin H$, then $
\R^d=H \oplus \textup{span}_{\R} \{c\}
$,
and each $x \in \R^d$ can be written as $x=x_H+\lambda_x c$ for a
unique $(x,\lambda_x) \in H \times \R$. Hence $H+c \Z=\{x:\lambda_x
\in \Z\}$ is closed by continuity of the projection $x \mapsto
\lambda_x$, and $\overline{G} \subseteq H+c \Z \neq \R^d$. 

\medskip

\noindent ($\Leftarrow$) \ Assume $\overline{G} \neq \R^d$. By Theorem
\ref{decom:opt},  $\overline{G}=V \oplus
\Lambda$ for a subspace $V$ and lattice
$\Lambda$ with $V \cap \textup{span}_\R \Lambda=\{0\}$. It
follows that the dimensions $n$ of $V$ and $m$ of the vector space
$\textup{span}_\R \Lambda$ satifsfy $n<d$ and $n+m \leq d$. If $m=0$, $G \subseteq V \subseteq H$ for some
codimension 1 space $H$. If $m \geq 1$, then
$\Lambda=\oplus_{i=1}^m a_i \Z$ for some basis $\{a_1,\dots,a_m\}$ of
$\textup{span}_\R \Lambda$. Let $W:=V \oplus
\textup{span}_\R \{a_i:i\neq m\}$ for $m>1$ and $W:=V$ for $m=1$. Then
$W$ is of dimension $n+m-1 \leq d-1$ and contained in some codimension 1 space
$H$. Hence $G \subseteq H+c \Z$ with $c=a_m$.
\end{proof}

\begin{proof}[Proof of Theorem \ref{thm:Liouville}]\

\smallskip

\noindent\eqref{label:thmLiou:c} $\Rightarrow$ \eqref{a4} \ If $u\in
L^\infty(\R^d)$ satisfy 
$\cL[u]=0$ in $\mathcal{D}'(\R^d)$, then $u$ is 
$\overline{G_\mu+W_{\sigma,b+c_\mu}}$-periodic  by Theorem 
\ref{thm:PeriodGeneralOp}. Hence $u$ is 
 constant 
by \eqref{label:thmLiou:c}.

\medskip

\noindent\eqref{a4} $\Rightarrow$ \eqref{label:thmLiou:c} \ Assume \eqref{label:thmLiou:c} does not hold  and let us construct a nontrivial $\overline{G_\mu+W_{\sigma,b+c_\mu}}$-periodic $L^\infty$-function.  By Corollary \ref{pro:group-multid}, 
\begin{equation}\label{cep2}
\overline{G_\mu+W_{\sigma,b+c_\mu}} \subseteq H+ c \Z,
\end{equation} 
 for some $c \in
\R^d$ and codimension 1 subspace $H \subset \R^d$. 
We can assume $c
\notin H$ since otherwise \eqref{cep2} will hold if we redefine $c$ to
be any element in $H^c$. As before, each $x \in \R^d$ can be written as $x=x_H+\lambda_x c$  for a unique pair $(x_H,\lambda_x) \in H \times \R$. Now let 
$
U(x):=\cos (2\pi  \lambda_x)
$
and note that for any $h \in H$ and ${ n } \in \Z$,
$$
x+h+{ n }c=\underbrace{(x_H+h)}_{\in H}+\underbrace{(\lambda_x+{ n })}_{\in \R} c,
$$
so that 
$$
U(x+h+{ n }c) =\cos (2\pi (\lambda_x+{ n }))=\cos (2\pi \lambda_x)= U(x).
$$
This proves that $U$ is $(H+c \Z)$-periodic and thus also
$\overline{G_\mu+W_{\sigma,b+c_\mu}}$-periodic.  By Theorem
\ref{thm:PeriodGeneralOp}, $\cL[U]=0$, and we have a nonconstant
counterexample of \eqref{a4}.  Note indeed that $u \in L^\infty(\R^d)$ since it is everywhere bounded by construction and $C^\infty$ (thus measurable) because the projection $x \mapsto \lambda_x$ is linear.  We therefore conclude that \eqref{a4}
implies \eqref{label:thmLiou:c} by contraposition.
\end{proof}

\section{Examples}\label{sec:examples}

Let us give examples for which the Liouville property  holds or fails. We will use Theorem \ref{thm:Liouville} or the following reformulation:  

\begin{corollary}\label{label:thmLiou:d}
Under the assumptions of Theorem \ref{thm:Liouville}, $\cL$ does \textup{not} satisfy the Liouville property if and only if 
\begin{equation}\label{failb}
\supp(\mu)+W_{\sigma,b+c_\mu} \subseteq H + c \Z, 
\end{equation}
for some codimension 1 subspace $H$ and vector $c$ of $\R^d$.
 \end{corollary}

\begin{proof}
Just note that $\overline{G(\supp(\mu)+W_{\sigma,b+c_\mu})}=\overline{G_\mu+W_{\sigma,b+c_\mu}}$ and apply Theorem \ref{thm:Liouville} and Corollary \ref{pro:group-multid}.
\end{proof}

\begin{example}\label{ex1}
\begin{enumerate}[{\rm (a)}]
\item For nonlocal operators $\cL=\cL^\mu$ with $\mu$ symmetric, \eqref{failb} reduces to 
\begin{equation}\label{fail}
\supp(\mu) \subseteq H+c \mathbb{Z},
\end{equation}
  for some $H$ of codimension $1$ and $c$. This fails  for fractional Laplacians, relativistic Schr\"odinger operators, convolution
operators, or most nonlocal operators appearing in finance whose L\'evy measures contain an open ball in their supports. In particular all these operators have the Liouville property.
\smallskip
\item Even if $\textup{supp} (\mu)$ has an empty interior,
  \eqref{fail} may fail and Liouville still hold. This is e.g. the case for the mean value operator
\begin{equation}\label{mean-value}
\mathcal{M}[u](x)=\int_{|z|=1} \big(u(x+z)-u(x)\big) \dd S(z),
\end{equation}
where $S$ denotes the $d-1$-dimensional surface measure.
\smallskip
\item We may have in fact the Liouville property with just a finite number of points in the support of $\mu$, see Example \ref{ex:finitenumberofpoints}. 
\smallskip
\item The way we have defined the nonlocal operator, if $\cL=\cL^\mu$ with general $\mu$, 
\eqref{failb} reduces to 
\begin{equation}
\label{failc}
\supp(\mu) \subseteq H+c \mathbb{Z} \quad \mbox{and} \quad c_\mu \in H,
\end{equation}
for some $H$ of codimension 1 and $c\in \R^d$. 
We can have  \eqref{fail}  without \eqref{failc} as e.g. for the 1--$d$ measure $\mu=\delta_{-1}+2\delta_{1}$. Indeed $\supp (\mu) \subset \Z$ but $c_\mu=1 \neq 0$. The associated operator $\cL^\mu$ then has the Liouville property even though it would not for any symmetric measure with the same support.
\smallskip
\item
A general operator $\cL=\cL^{\sigma,b}+\cL^\mu$ may satisfy the Liouville property even though each part $\cL^{\sigma,b}$ and $\cL^{\mu}$ does not. 
A simple 3--$d$ example is given by $\cL=\partial_{x_1}^2+\partial_{x_2}+(\partial_{x_3}^2)^{\alpha}$, $\alpha \in (0,1)$. 

Indeed $\sigma=(1,0,0)^\texttt{T}$, $b=(0,1,0)$, $\dd \mu(z)=\frac{c(\alpha) \dd z_3}{|z_3|^{1+2 \alpha}}$ with $c(\alpha)>0$, thus $c_\mu=0$, $W_{\sigma,b}=\R \times \R \times \{0\}$, and $G_\mu=\{0\} \times \{0\}\times \R$, so the result follows from Theorem \ref{thm:Liouville}.
\item For other kinds of interactions between the local and nonlocal parts, see Example \ref{ex:kro}.
\end{enumerate}
\end{example}

\begin{remark}
The Liouville property for the nonlocal operator \eqref{mean-value}
implies the classical Liouville result for the Laplacian, since
$\mathcal{M}[u]=0$ for harmonic functions~$u$.
\end{remark}

In the 1--$d$ case, the general form of the operators which do not satisfy the Liouville property is very explicit. 

\begin{corollary}
Assume $d=1$ and $\cL:C^\infty_\textup{c} (\R) \to C(\R)$ is a linear translation invariant operator satisfying the maximum principle \eqref{mp}. Then the following statements are
equivalent:
\begin{enumerate}[{\rm (a)}]
  \smallskip
\item\label{1da} There are nonconstant $u\in L^\infty(\R)$ satisfying $\cL[u]=0$ in
  $\mathcal{D}'(\R)$.
   \smallskip
\item\label{label:thmLiou:cb} There are $g> 0$ and a nonnegative $\{\omega_n\}_{n} \in l^1(\Z)$ such that
\begin{equation*}
\cL[u](x)=\sum_{n \in \mathbb{Z}}(u(x+n g)-u(x)) \omega_n .
\end{equation*}
\end{enumerate}
\end{corollary}

\begin{proof} 
If \eqref{label:thmLiou:c} holds, any $g$-periodic
  function satisfies $\cL[u]=0$ in $\R$. Conversely, if \eqref{1da}
  holds then $\cL$ is of the form
  \eqref{eq:GenOp1}--\eqref{def:localOp1}--\eqref{def:levy1} by
  \cite{Cou64}. By Corollary \ref{label:thmLiou:d}, there is $g \geq0$
  such that $\textup{supp} (\mu)+W_{\sigma,b+c_\mu} \subseteq g
  \mathbb{Z}$. In particular $\sigma=b+c_\mu=0$ and $\mu$ is a a sum of
  Dirac measures: $\mu=\sum_{n \in \mathbb{Z}} \omega_n \delta_{n g}$.\footnote{If $g=0$ then $\mu=0$ and the rest of the proof is trivial.} By \eqref{as:mus}, each $\omega_n \geq 0$ and $\sum_{n \in \mathbb{Z}} \omega_n<\infty$. Injecting these facts into \eqref{eq:GenOp1}--\eqref{def:localOp1}--\eqref{def:levy1}, we can easily rewrite $\cL$ as in \eqref{label:thmLiou:cb}.
\end{proof}

\begin{example}\label{ex2}
\begin{enumerate}[{\rm (a)}]
\item In 1--$d$, the Liouville property holds for any nontrivial operator with nondiscrete L\'evy measure.
\item For discrete L\'evy measures, we need $\sigma \neq 0$ or $b \neq
  -c_\mu$ or $G_\mu=\R$ for Liouville to hold. The condition
  $G_\mu=\R$ is typically satisfied if $\overline{\textup{supp} (\mu)}^{\R}$ 
  has an accumulation point or if $\textup{supp} (\mu)$ contains two points  $z_1,z_2$  with
  irrationial ratio  $\frac{z_1}{z_2}$  (see Theorem
  \ref{thm:CharKron}).  Another example is when $\textrm{supp} (\mu) =\{\frac{n^2+1}{n}\}_{n \geq 1}$, which has no accumulation point or contains any pair with irrational ratio.
\end{enumerate}
\end{example}

Let us continue with interesting consequences of the Kronecker theorem
on Diophantine approximation (p. 507 in  \cite{GoMo16}).

\begin{theorem}[Kronecker theorem]\label{thm:CharKron}
Let $c=(c_1,\dots,c_d)\in \R^d$. Then $\overline{c \Z+\Z^d}=\R^d$ if and only if 
$\{1,c_1,\dots,c_d\}$ is linearly independent over $\Q$.
\end{theorem}

We can use this result to get the Liouville property with just a finite number of points in the support of the L\'evy measure. 
\begin{example}\label{ex:finitenumberofpoints}
\begin{enumerate}[{\rm (a)}]
\item Consider the operator
\[
\cL[u](x)=u(x+c) +\sum_{i=1}^d u(x+e_i) - (d+1) u(x)
\]
for some $c=(c_1,\ldots,c_d)\not=0$ where $\{e_1,\dots,e_d\}$ is the canonical basis.  Liouville holds if and only if $\{1,c_1,\ldots,c_d\}$ is linearly independent over $\Q$. Indeed $G_\mu=\overline{c \Z+\Z^d}$, so the result follows from Theorems \ref{thm:Liouville} and \ref{thm:CharKron}.
\item For more general operators $\cL[u](x)=\sum_{z \in S}
  (u(x+z)-u(x))\omega(z)$, with $S$ finite and $\omega(\cdot)>0$, we
  may have similar results by applying Theorem \ref{thm:CharKron} (or
  variants) and changing coordinates.
\end{enumerate}
\end{example}

Let us end with an illustration of how the local part may interact
with such nonlocal operators. We give 2--$d$ examples of the form
$$
\cL[u](x)=\tilde{b}_1u_{x_1}+\tilde{b}_2u_{x_2}+u(x+z_1)+u(x+z_2)-2 u(x)
$$
where $\tilde{b}$ represents the full drift $b+c_\mu$.
\begin{example}\label{ex:kro}
\begin{enumerate}[\rm (a)]
\item If $\tilde{b}, z_1, z_2$ are collinear, Liouville does not hold by Theorem~\ref{thm:Liouville}.
\item If  $z_1$ and $z_2$ are collinear and linearly independent
  of $\tilde{b}$ as in
\begin{equation*}
\cL[u](x)=u_{x_1}(x)+u(x_1,x_2+\alpha)+u(x_1,x_2+\beta)-2 u(x),
\end{equation*}
then the Liouville property holds if and only if $\frac{\alpha}{\beta} \notin
\Q$. 

Indeed, here we have $G_{\mu}=\{0\} \times \overline{\alpha \Z+\beta \Z}$ and
$\textup{span}_\R \{b+c_\mu=(1,0)\}=\R \times \{0\}$, so we conclude by Theorems \ref{thm:Liouville} and~\ref{thm:CharKron}.
\item If $\{z_1,z_2\}$ is a basis of $\R^2$ as in
\begin{equation*}
\begin{split}
\hspace{6mm} \cL[u](x)  =&\tilde{b}_1 u_{x_1}(x)+\tilde{b}_2 u_{x_2}(x)+u(x_1+1,x_2)+u(x_1,x_2+1)-2 u(x),
\end{split}
\end{equation*}
then Liouville holds if and only if $\tilde{b}_1\not=0$ and
$\frac{\tilde{b}_2}{\tilde{b}_1} \notin \Q$.

Indeed, let us define $G:=G_\mu+W_{\sigma,b+c_\mu}$ where
we note that $G_\mu=\Z^2$ and
$W_{\sigma,b+c_\mu}=\textup{span}_{\R}\{(\tilde{b}_1,\tilde{b}_2)\}$. If
$\tilde{b}_1=0$ or $\tilde{b}_2=0$, then $\overline{G}\subseteq \Z \times \R$ or $\R
\times \Z$ which is not $\R^2$. Assume now that
$\tilde{b}_1,\tilde{b}_2\not=0$ and $\frac{\tilde{b}_2}{\tilde{b}_1}
\in \Q$, i.e., $\frac{\tilde{b}_2}{\tilde{b}_1}=\frac{p}{q}$ with $p,q
\neq 0$. Then
$$G\subseteq T:=\Big(\frac{1}{p},0\Big)\Z +
\textup{span}_{\R}\Big\{\Big(1,\frac{\tilde{b}_2}{\tilde{b}_1}\Big)\Big\}=\Big\{\Big(\frac{k}{p}+r,
r\frac{p}{q}\Big): k \in \Z, \ r \in \R\Big\}$$
since
$\textup{span}_{\R}\{(\tilde{b}_1,\tilde{b}_2)\}=\textup{span}_{\R}\{(1,\frac{\tilde{b}_2}{\tilde{b}_1})\}\subset
T$ and $\Z^2\subset T$. The last statement follows since for any
$(m,n)\in \Z^2$, we can take $k=pm-qn\in \Z$ and $r=n \frac{q}{p}\in
\R$. Since $\overline{T}\neq \R^2$, Liouville does not hold by Theorem \ref{thm:Liouville} and Corollary
\ref{pro:group-multid}. 

Conversely, assume $\tilde{b}_1,\tilde{b}_2\not=0$ and
$\frac{\tilde{b}_2}{\tilde{b}_1} \notin \Q$. Then $(0,\frac{\tilde{b}_2}{\tilde{b}_1})=(-1,0)+(1,\frac{\tilde{b}_2}{\tilde{b}_1})\in G$ and since $(0,1) \in G$, we get that $\{0\} \times (\Z+\frac{\tilde{b}_2}{\tilde{b}_1} \Z) \subset G$.
By Theorem \ref{thm:CharKron}, $\{0\} \times \R \subset
\overline{G}$. Arguing similarly with
$(\frac{\tilde{b}_1}{\tilde{b}_2},0)$, we find that $\R \times
\{0\}\subset \overline{G}$. Hence $\overline{G}=\R^2$ and
Liouville holds by Theorem \ref{thm:Liouville}.
\end{enumerate}
\end{example}

\let\oldaddcontentsline\addcontentsline
\renewcommand{\addcontentsline}[3]{}
\section*{Acknowledgements}
F.d.T., J.E.,  and E.R.J. were  supported  by the Toppforsk (research excellence) project Waves and Nonlinear Phenomena (WaNP), grant no. 250070 from the Research Council of Norway; F.d.T. also by the Basque Government through the BERC 2018-2021 program, and the Spanish Ministry of Science, Innovation and Universities: BCAM Severo Ochoa accreditation SEV-2017-0718, and the Spanish research project PGC2018-094522-B-I00.

F.d.T. and J.E. are grateful to Laboratoire de Math\'ematiques de Besan\c{c}on (LMB, UBFC) and Ecole Nationale Sup\'erieure de M\'ecanique et des Microtechniques (ENSMM) for hosting them during their visit in May  2018.

During the final preparation of this paper, we appreciated the
feedback from the community which helped us to put the paper in
context and also to improve the presentation.

\let\oldaddcontentsline\addcontentsline
\renewcommand{\addcontentsline}[3]{}


\begin{thebibliography}{10}

\bibliofont


\bibitem{A-VMaRoT-M10}
F.~Andreu-Vaillo, J.~M. Maz{\'o}n, J.~D. Rossi, and J.~J. Toledo-Melero.
\newblock {\em Nonlocal diffusion problems}, volume 165 of {\em Mathematical
  Surveys and Monographs}.
\newblock American Mathematical Society, Providence, RI; Real Sociedad
  Matem\'atica Espa\~nola, Madrid, 2010.

\bibitem{BaBaGu00}
M.~T. Barlow, R.~F. Bass, and C.~Gui.
\newblock The {L}iouville property and a conjecture of {D}e {G}iorgi.
\newblock {\em Comm. Pure Appl. Math.}, 53(8):1007--1038, 2000.

\bibitem{BeHaRo07}
H.~Berestycki, F.~Hamel, and L.~Rossi.
\newblock Liouville-type results for semilinear elliptic equations in unbounded
  domains.
\newblock {\em Ann. Mat. Pura Appl. (4)}, 186(3):469--507, 2007.

\bibitem{BoKuNo02}
K.~Bogdan, T.~Kulczycki, and A.~Nowak.
\newblock Gradient estimates for harmonic and {$q$}-harmonic functions of
  symmetric stable processes.
\newblock {\em Illinois J. Math.}, 46(2):541--556, 2002.

\bibitem{BrChQu12}
C.~Br\"andle, E.~Chasseigne, and F.~Quir\'os.
\newblock Phase transitions with midrange interactions: a nonlocal {S}tefan
  model.
\newblock {\em SIAM J. Math. Anal.}, 44(4):3071--3100, 2012.

\bibitem{BrCo18}
J.~Brasseur and J.~Coville.
\newblock A counterexample to the {L}iouville property of some nonlocal
  problems.
\newblock Preprint, arXiv:1804.07485v1 [math.AP], 2018.

\bibitem{BrCoHaVa19}
J.~Brasseur, J.~Coville, F.~Hamel, and E.~Valdinoci.
\newblock Liouville type results for a nonlocal obstacle problem.
\newblock {\em Proc. London Math. Soc.}, 119(2):291--328, 2019.

\bibitem{CaSi14}
X.~Cabr\'e and Y.~Sire.
\newblock Nonlinear equations for fractional {L}aplacians, {I}: {R}egularity,
  maximum principles, and {H}amiltonian estimates.
\newblock {\em Ann. Inst. H. Poincar\'e Anal. Non Lin\'eaire}, 31(1):23--53,
  2014.

\bibitem{ChDALi15}
W.~Chen, L.~D'Ambrosio, and Y.~Li.
\newblock Some {L}iouville theorems for the fractional {L}aplacian.
\newblock {\em Nonlinear Anal.}, 121:370--381, 2015.


\bibitem{CD60}
  G.~Choquet and J.~Deny.
  \newblock Sur l'\'equation de convolution {$\mu =\mu \ast \sigma $}.
  \newblock {\em C. R. Acad. Sci. Paris} 250: 799–801, 1960.
  
\bibitem{Cio12}
A.~Ciomaga.
\newblock On the strong maximum principle for second-order nonlinear parabolic
  integro-differential equations.
\newblock {\em Adv. Differential Equations}, 17(7-8):635--671, 2012.

\bibitem{CoTa04}
R.~Cont and P.~Tankov.
\newblock {\em Financial modelling with jump processes}.
\newblock Chapman \& Hall/CRC Financial Mathematics Series. Chapman \&
  Hall/CRC, Boca Raton, FL, 2004.

\bibitem{Cou64}
P.~Courr\`ege.
\newblock G\'en\'erateur infinit\'esimal d'un semi-groupe de convolution sur
  {$R^{n}$}, et formule de {L}\'evy-{K}hinchine.
\newblock {\em Bull. Sci. Math. (2)}, 88:3--30, 1964.

\bibitem{Cov08}
J.~Coville.
\newblock Remarks on the strong maximum principle for nonlocal operators.
\newblock {\em Electron. J. Differential Equations}, pages No. 66, 10, 2008.

\bibitem{DTEnJa17b}
F.~del Teso, J.~Endal, and E.~R. Jakobsen.
\newblock On distributional solutions of local and nonlocal problems of porous
  medium type.
\newblock {\em C. R. Math. Acad. Sci. Paris}, 355(11):1154--1160, 2017.

\bibitem{DTEnJa17a}
F.~del Teso, J.~Endal, and E.~R. Jakobsen.
\newblock Uniqueness and properties of distributional solutions of nonlocal
  equations of porous medium type.
\newblock {\em Adv. Math.}, 305:78--143, 2017.

\bibitem{DTEnJa18b}
F.~del Teso, J.~Endal, and E.~R. Jakobsen.
\newblock Robust numerical methods for nonlocal (and local) equations of porous
  medium type. {P}art {II}: {S}chemes and experiments.
\newblock {\em SIAM J. Numer. Anal.}, 56(6):3611--3647, 2018.

\bibitem{Fal15}
M.~M. Fall.
\newblock Entire {$s$}-harmonic functions are affine.
\newblock {\em Proc. Amer. Math. Soc.}, 144(6):2587--2592, 2016.

\bibitem{FaWe16}
M.~M. Fall and T.~Weth.
\newblock Liouville theorems for a general class of nonlocal operators.
\newblock {\em Potential Anal.}, 45(1):187--200, 2016.

\bibitem{Far07}
A.~Farina.
\newblock Liouville-type theorems for elliptic problems.
\newblock In {\em Handbook of differential equations: stationary partial
  differential equations. {V}ol. {IV}}, Handb. Differ. Equ., pages 61--116.
  Elsevier/North-Holland, Amsterdam, 2007.

\bibitem{GoMo16}
S.~M. Gonek and H.~L. Montgomery.
\newblock Kronecker's approximation theorem.
\newblock {\em Indag. Math. (N.S.)}, 27(2):506--523, 2016.

\bibitem{HuDuWu18}
Q.~Huang, J.~Duan, and J.-L. Wu.
\newblock Maximum principles for nonlocal parabolic {W}aldenfels operators.
\newblock {\em Bulletin of Mathematical Sciences}, 2018.

\bibitem{Lan72}
N.~S. Landkof.
\newblock {\em Foundations of modern potential theory}.
\newblock Springer-Verlag, New York-Heidelberg, 1972.

\bibitem{Mar03}
J.~Martinet.
\newblock {\em Perfect lattices in {E}uclidean spaces}, volume 327 of {\em
  Grundlehren der Mathematischen Wissenschaften [Fundamental Principles of
  Mathematical Sciences]}.
\newblock Springer-Verlag, Berlin, 2003.

\bibitem{Miy15}
Y.~Miyazaki.
\newblock Liouville's theorem and heat kernels.
\newblock {\em Expo. Math.}, 33(1):101--104, 2015.

\bibitem{Nel61}
E.~Nelson.
\newblock A proof of {L}iouville's theorem.
\newblock {\em Proc. Amer. Math. Soc.}, 12:995, 1961.

\bibitem{PrZa04}
E.~Priola and J.~Zabczyk.
\newblock Liouville theorems for non-local operators.
\newblock {\em J. Funct. Anal.}, 216(2):455--490, 2004.

\bibitem{R-OSe16a}
X.~Ros-Oton and J.~Serra.
\newblock Boundary regularity for fully nonlinear integro-differential
  equations.
\newblock {\em Duke Math. J.}, 165(11):2079--2154, 2016.

\bibitem{Ros09}
L.~Rossi.
\newblock Liouville type results for periodic and almost periodic linear
  operators.
\newblock {\em Ann. Inst. H. Poincar\'{e} Anal. Non Lin\'{e}aire},
  26(6):2481--2502, 2009.

\bibitem{ScWa12}
R.~L. Schilling and J.~Wang.
\newblock On the coupling property and the {L}iouville theorem for
  {O}rnstein-{U}hlenbeck processes.
\newblock {\em J. Evol. Equ.}, 12(1):119--140, 2012.

\end{thebibliography}
\end{document}